\documentclass[11pt]{article}
\usepackage{setspace}
\usepackage{enumerate}
\usepackage{amsfonts}
\usepackage{amsmath}
\usepackage{bm}
\usepackage{amsthm}
\usepackage[text={150mm,220mm}]{geometry}
\usepackage{xcolor}
\usepackage[colorlinks,
            linkcolor=blue,
            anchorcolor=blue,
            citecolor=blue
            ]{hyperref}

\title{The Limiting Behavior of the FTASEP with Product Bernoulli Initial Distribution}
\author{Dayue Chen\footnote{Department of Probability and Statistics, School of Mathematical Sciences, Peking University, Bei-
jing, China, {\it E-mail}: dayue@pku.edu.cn},~~  Linjie Zhao\footnote{{\it E-mail}: zhaolinjie@pku.edu.cn}}
\date{\tiny \today}

\begin{document}

\setlength{\baselineskip}{18pt}

\maketitle

\begin{abstract}
We study the facilitated totally asymmetric exclusion process on the one dimensional integer lattice. We investigate the invariant measures and the limiting behavior of the process. We mainly derive the limiting distribution of the process when the initial distribution is the Bernoulli product measure with density $1/2$. We also prove that in the low density regime, the system finally converges to an absorbing state.

\begin{flushleft}
\end{flushleft}
\end{abstract}

\newtheorem{thm}{Theorem}
\newtheorem{pro}[thm]{Proposition}
\newtheorem{cor}[thm]{Corrolary}
\newtheorem{defi}[thm]{Definition}
\newtheorem{lem}[thm]{Lemma}
\newtheorem{conjec}[thm]{Conjecture}

\section{Introduction}

The exclusion process plays the role of a paradigm in non-equilibrium statistical mechanics. The model describes many systems, including traffic \cite{KlauckSchadschneider99}, ionic conductor \cite{Richards77}, and RNA transcription \cite{MacDonald68}. Mathematically, it was first introduced by Spitzer \cite{spitzer70} as a model of a lattice gas. The exclusion process has been extensively studied since then (see e.g. \cite{liggett85,liggett99}). The dynamics is as follows: Particles do simple random walks according to the exclusion rule, which means that there is at most one particle per site.

A particular case of one dimensional exclusion process is the totally asymmetric simple exclusion process (TASEP). Independently, each particle jumps with rate $1$ to the right neighboring site, provided it is empty. Despite its simple structure, the TASEP is closely related to the corner growth model \cite{Rost81} and the KPZ theory \cite{Johansson00}.

In this paper, we consider the facilitated totally asymmetric simple exclusion process (\mbox{FTASEP}) on $\mathbb{Z}$, a variation of the TASEP. The model is defined with a dynamics where a particle from an occupied site hops to the right neighboring vacant site stochastically if the left one is occupied. The state space of the FTASEP is $\mathbf{X} = \{0,1\}^{\mathbb{Z}}$. For $\eta \in \mathbf{X}$, $\eta (x) = 1$ means that there is a particle at site $x$ while $\eta (x) = 0$ means that there is a hole at site $x$. The infinitesimal generator of the FTASEP is
\begin{equation}
\mathcal{L} f (\eta) = \sum_{x \in \mathbb{Z}} \eta (x-1) \eta (x) [1 - \eta (x+1)] [f (\eta^{x,x+1}) - f (\eta)],
\end{equation}
where $f$ is a cylinder function (depending  on  only finitely many coordinates) and
\begin{equation}
\eta^{x,x+1} (y) =
\begin{cases}
\eta (x+1) &\text{if $y = x$},\\
\eta (x) &\text{if $y = x+1$},\\
\eta (y) &\text{otherwise}.\\
\end{cases}
\end{equation}

The model describes the motion in glasses: The particle moves slower as the local density increases (the exclusion rule), but needs a stimulus to move (the facilitated rule). It was first introduced in \cite{RossiPastorVespignani00} as a conserved lattice gas model which undergoes a continuous phase transition to an absorbing state at a critical value of the  particle density.  The model has been studied analytically and numerically in the physics literature \cite{BasuMohanty09,GabelKrapivskyRedner10}, and recently has got mathematicians' attention from different aspects \cite{BaikBarraquandCorwinToufic17,Blondel etal18}.

An important issue is to consider the invariant measures of the FTASEP. While it is hard to characterise all of the invariant measures for the process, there exist nontrivial invariant measures when the density is above $1/2$. For the process on the one dimensional ring, when the particle density $\rho > 1/2$, the uniform measure on the family of maximal-island configurations (the configurations with no adjacent zeros) is the unique invariant measure \cite{GabelKrapivskyRedner10}. The invariant measure of the model  can also be calculated by using the matrix product ansatz \cite{BasuMohanty09}.   Our first result shows that depending on whether the density $\rho > 1/2$ or not, the invariant measures of the process on $\mathbb{Z}$ exhibit different behaviors.

We now discuss a little ergodic theory (see e.g. Section 4, Chapter 1, \cite{liggett85}). For $x \in \mathbb{Z}$, define the shift transformation $\tau_x$ on $\mathbf{X}$ by
$$
( \tau_x \eta ) (y) = \eta (x + y).
$$
In a natural way, these induce shift transformations on the space of all functions  on  $\mathbf{X}$ via
$$
( \tau_x f ) (\eta) = f ( \tau_x \eta ),
$$
and then on the space of distributions on  $\mathbf{X}$ via
$$
\int f d (\tau_x \mu) = \int (\tau_x f) d \mu.
$$
The distribution $\mu$ on $\mathbf{X}$ is said to be translation invariant if $\tau_x \mu = \mu$ for all $x \in \mathbb{Z}$. The translation invariant measure $\mu$ is said to be (spatial) ergodic if whenever $\tau_x f = f$ for all $x \in \mathbb{Z}$, it follows that $f$ is constant a.s. relative to $\mu$.

For a translation invariant  measure $\mu$ on $\mathbf{X}$, we say the measure $\mu$ has density $\rho$ if
\begin{equation}
\mu \{\eta: \eta (x) = 1 \} = \rho,
\end{equation}
and we say $\mu$ is {\it degenerate} if
\begin{equation}
\mu \{\eta: \eta (x) = \eta (x+1) = 1 \} = 0,
\end{equation}
where the left-hand sides of the above two equations do not depend on $x$ by translation invariance. Note that if $\eta$ is sampled from a degenerate distribution $\mu$ and the process is started from $\eta$, then the process will be trapped in the configuration $\eta$ forever.

\begin{thm}\label{thm:properties4}
(1) For $1/2 < \rho < 1$, there exists a family of (spatial) ergodic non-degenerate measures invariant for the process with density $\rho$.\\
(2) For $0 < \rho \leq 1/2$, there are no (spatial) ergodic non-degenerate measures invariant for the process with density $\rho$.
\end{thm}

\noindent {\bf Remark.} For $1/2 < \rho < 1$, the existence of such measures was observed in \cite{BaikBarraquandCorwinToufic17}. Our contribution is to prove the second statement of the theorem.

Another question is to consider the behavior of the process when the initial distribution of the process is the Bernoulli product measure with density $\rho$. For the facilitated exclusion on the one dimensional ring, it is easy to see that when the density of particles is below $1/2$, then the system will finally converges to an absorbing state \cite{BasuMohanty09}. Our second result shows that this is also the case for the  process on $\mathbb{Z}$.

To be specific, define the {\it freezing time} $\tau_x$ at site $x$ by
\begin{equation}
\tau_x := \inf \{s: \eta_t (x) = \eta_s (x), \forall t \geq s\}.
\end{equation}
Let $\nu_\rho,  \rho \in [0,1]$, be the product measure on $\{0,1\}^{\mathbb{Z}}$ with marginal given by
\begin{equation}
\nu_\rho \{ \eta: \eta (x) = 1 \} = \rho,~~x \in \mathbb{Z}.
\end{equation}
For a probability measure $\mu$ on $\mathbf{X}$, by $\mathbf{P}_{\mu}$ and $\mathbf{E}_\mu$ denote the probability measure and the expectation on the space $D ([0,\infty),\mathbf{X})$ induced by the FTASEP $\eta_t$ and the initial distribution of the process $\mu$.

\begin{thm}\label{thm:properties1}
For the FTASEP on $\mathbb{Z}$, the following dichotomy holds:
\begin{equation}
\mathbf{P}_{\nu_\rho} (\tau_0 < \infty) =
\begin{cases}
1 &\text{if $0 < \rho < 1/2$},\\
0 &\text{if $1/2 \leq \rho < 1$}.
\end{cases}
\end{equation}
\end{thm}

In the critical case $\rho = 1/2$, we calculate the limiting distribution explicitly.  Let ${\bm \eta^1}$ (${\bm \eta^0}$ resp.) be the {\it alternative configuration} with particles placed on odd (even resp.) sites: ${\bm \eta^1} (x) = 1$ (${\bm \eta^0} (x) = 1$ resp.) if and only if $x$ is odd (even resp.). For $\eta \in \mathbf{X}$, let $\delta_\eta$ be the Dirac measure concentrated on the configuration $\eta$. Denote by $S (t)$ the semigroup corresponding to the generator $\mathcal{L}$, then $\mu S (t)$ is the distribution of the process at time $t$ when the initial distribution is $\mu$.

\begin{thm}\label{thm:properties3}
For the FTASEP on $\mathbb{Z}$,
\begin{equation}
\lim_{t \rightarrow \infty} \nu_{1/2} S(t) = \frac{1}{2} \delta_{{\bm \eta^1}} + \frac{1}{2} \delta_{{\bm \eta^0}}.
\end{equation}
\end{thm}

\noindent {\bf Remark.} (1) Theorem \ref{thm:properties1} says that in the critical regime, the system remains active forever. However, the system is absorbed finally in the sense of distribution by the above theorem.
(2) In the subcritical case, we show how to derive the limiting distribution by induction in Section \ref{sec5:properties}, while we are not able to write down the limiting distribution explicitly.
(3) The limiting distribution of the process in the supercritical regime needs to be further investigated.

The paper is organised as follows.  In Section \ref{sec2:properties}, we mainly prove the second statement of Theorem \ref{thm:properties4} by generator calculation and the ergodicity of the measure. Theorem \ref{thm:properties1} is proved in Section \ref{sec3:properties}. The idea is to consider the height process of the exclusion process (see e.g. \cite{PrahoferSpohn00}).  The conception of record plays an important role. The definition of record for a sequence of random variables can be found in Example 2.3.2. \cite{Durrett}. In Section \ref{sec4:properties}, we  first consider the facilitated exclusion process viewed from a tagged hole. This process can be mapped into the zero range process, which was used early in \cite{Kipnis86} to prove the central limit theorem for the tagged particle in the simple exclusion process. We then can easily prove Theorem \ref{thm:properties3} by the symmetry of the process. In Section \ref{sec5:properties}, we show how to derive the limiting distribution in the subcritical case.

\section{Proof of Theorem \ref{thm:properties4}}\label{sec2:properties}

In this Section, we prove Theorem \ref{thm:properties4}. As stated before, it was observed in \cite{BaikBarraquandCorwinToufic17} that there exists a family of (spatial) ergodic non-degenerate invariant measures for the process with density $1/2 < \rho < 1$. For completeness, we show how to construct such measures.

\begin{proof}[Proof of Theorem \ref{thm:properties4}]
(1) The basic idea is that we can view the facilitated exclusion particles between successive holes as the zero range particles at a site. Since the zero range process possesses a family of non-trivial invariant measures, so does the facilitated exclusion process.

Let $\mathbf{X_0}$ be the family of configurations without zero pairs,
\begin{equation}
\mathbf{X_0} = \{\eta \in \mathbf{X}: (\eta (x), \eta (x+1)) \neq (0,0) ~~\text{for all $x$} \}.
\end{equation}
Note that $\mathbf{X_0}$ is a closed set for the process $\eta_t$. For each $\rho \in (1/2,1)$, we will define a  renewal measure $\mu_\rho$ on the configuration space $\mathbf{X_0}$. To be specific, for $A = \{x_1, x_2, \ldots, x_n\}$ with $x_i + 1 < x_{i+1}$ for all $1 \leq i \leq n-1$, define $\mu_\rho$ by
\begin{equation}\label{eqn3:properties}
\begin{split}
\mu_\rho (\eta(x_i) &= 0 ~\text{for $1 \leq i \leq n$ and}~ \eta (x) = 1 ~\text{for all $x \notin A$ such that $x_1 < x < x_n$})\\
&= (1 - \rho) \prod_{i = 1}^{n - 1} \left( \varphi (\rho)^{x_{i+1} - x_i - 2} (1 - \varphi (\rho))\right),
\end{split}
\end{equation}
i.e., the number of particles between successive holes has geometric distribution with parameter $1 - \varphi (\rho)$. Since the average number of holes is $1 - \rho$, $\varphi$ and $\rho$ are related in the following way:
\begin{equation}
\frac{1}{1 - \varphi} + 1 = \frac{1}{1 - \rho}.
\end{equation}
Theofore, $\varphi (\rho) = \frac{2 \rho - 1}{\rho}$. The fact that $\mu_\rho$ is an invariant measure of the FTASEP was observed in \cite{BaikBarraquandCorwinToufic17}.

Next we show the ergodicity of $\mu_\rho$. For each measurable set $A$ such that $\mu_\rho \left( (\tau_x^{-1} A) \Delta A \right) = 0 $ for all $x \in \mathbb{Z}$, we need to show that $\mu_\rho (A) = 0$ or $1$. For each $\epsilon > 0$, there exists a set $B$ depending on only finitely many coordinates such that
\begin{equation}
\mu_\rho (A \Delta B) \leq \epsilon.
\end{equation}
Then for any $x$,
\begin{equation}\label{eqnA:properties}
  \begin{split}
  | \mu_\rho ( A) - \mu_\rho \left(B \cap (\tau_x^{-1} B) \right) | &= | \mu_\rho \left( A \cap (\tau_x^{-1} A) \right) - \mu_\rho \left(B \cap (\tau_x^{-1} B) \right) | \\
  &\leq  \mu_\rho \left( (A \cap (\tau_x^{-1} A)) \Delta (B \cap (\tau_x^{-1} B)) \right)\\
  &\leq \mu_\rho (A \Delta B) + \mu_\rho \left( (\tau_x^{-1} A) \Delta (\tau_x^{-1} B) \right) \leq 2 \epsilon,
  \end{split}
\end{equation}
and
\begin{equation}\label{eqnB:properties}
\begin{split}
|\mu_\rho (B) \mu_\rho (\tau_x^{-1} B) &- \mu_\rho (A) \mu_\rho (\tau_x^{-1} A)| \\
 &\leq |\mu_\rho (B) - \mu_\rho (A)| \mu_\rho (\tau_x^{-1} B)  + |\mu_\rho (\tau_x^{-1} B) - \mu_\rho (\tau_x^{-1} A)| \mu_\rho (A) \leq 2 \epsilon.
\end{split}
\end{equation}
To bound the term $\mu_\rho \left(B \cap (\tau_x^{-1} B) \right) - \mu_\rho (B) \mu_\rho (\tau_x^{-1} B)$, we insert the following lemma.

\begin{lem}
For any cylinder functions $f$ and $g$, as $x \rightarrow \infty$,
\begin{equation}
\int f (\eta) g (\tau_x \eta) \mu_\rho (d \eta) \rightarrow \int f d \mu_\rho \int g d \mu_\rho.
\end{equation}
\end{lem}

\begin{proof}
Since $f$ and $g$ depend on only finitely many coordinates, there exist $x_1 < x_2 < \ldots < x_n$ such that $f (\eta) = f (\eta (x_1), \ldots, \eta(x_n))$ and similarly $g (\eta) = g (\eta (x_1), \ldots, \eta(x_n))$. Under $\mu_\rho$, $\{\eta (x), x \geq x_1\}$ is a $\{0,1\}$-valued Markov chain with transition probability $p (0,1) = 1, p (1,1) = (2 \rho - 1) / \rho$ and $p (1,0) = (1 - \rho) / \rho$, and with invariant distribution $\pi$ such that $\pi (1) = \rho$ and $\pi (0) = 1 - \rho$. Denote by $\mathcal{E}$ the law of the Markov chain with initial distribution $\pi$. Then for large enough $x$ such that $x + x_1 > x_n$, by the Markov property,
\begin{equation}
\begin{split}
\mathcal{E} &\left[ f (\eta (x_1), \ldots, \eta(x_n)) g (\eta (x_1 + x), \ldots, \eta(x_n + x)) \right] \\
&= \mathcal{E} \left\{ f (\eta (x_1), \ldots, \eta(x_n))  \mathcal{E} [ g (\eta (0), \ldots, \eta(x_n - x_1)) | \eta (x_1 + x) ] \right\}.
\end{split}
\end{equation}
The right-hand side of the above identity converges to $\mathcal{E} \left[ f (\eta (x_1), \ldots, \eta(x_n)) \right] \mathcal{E} \left[ g (\eta (x_1), \ldots, \eta(x_n)) \right]$, and the lemma follows.
\end{proof}

We return to the proof of the ergodicity of $\mu_\rho$. By the above lemma, we can find $x$ large enough such that
\begin{equation}\label{eqnC:properties}
|\mu_\rho \left(B \cap (\tau_x^{-1} B) \right) - \mu_\rho (B) \mu_\rho (\tau_x^{-1} B)| \leq \epsilon.
\end{equation}
By (\ref{eqnA:properties}), (\ref{eqnB:properties}) and (\ref{eqnC:properties}), $| \mu_\rho ( A) -  \mu_\rho ( A)^2| \leq 5 \epsilon$. Since $\epsilon$ is arbitrary, $\mu_\rho ( A) \in \{0,1\}$.

(2) Let $\mu$ be a translation invariant and (spatial) ergodic measure invariant for the process with density $0 < \rho \leq 1/2$. Our goal is to show $\mu$ is degenerate, i.e.,
\begin{equation}\label{eqn7:properties}
\mu \{\eta: \eta (x) = \eta (x+1) = 1 \} = 0.
\end{equation}
For simplicity, by $\mu (11)$ denote $\mu \{\eta: \eta (x) = \eta (x+1) = 1\}$. Then
\begin{equation}
\mu (00) = \mu (0) - \mu (10) \geq \mu (1) - \mu (10) = \mu (11).
\end{equation}
If $\mu (00) = 0$, then $\mu (11) = 0$, and (\ref{eqn7:properties}) holds. Now assume $\mu (00) > 0$.

By generator calculation we can get some properties of $\mu$. Define $f (\eta) = 1 \{\eta (x) = \eta (x+1) = 0\}$. Since $\mu$ is invariant for the process,
\begin{equation}\label{eqn5:properties}
 \int \mathcal{L} f (\eta) \mu (d \eta) = 0.
\end{equation}
By direct calculation,
\begin{equation}
\mathcal{L} f (\eta) = - 1 \{\eta (x - 2) = \eta (x - 1) = 1, \eta (x) = \eta (x+1) = 0\}.
\end{equation}
Then by (\ref{eqn5:properties}),
\begin{equation}
  \mu (1100) = 0.
\end{equation}
Again, let $f (\eta) = 1 \{\eta (x - 2) = \eta (x - 1) = 1, \eta (x) = \eta (x+1) = 0\}$, then $\mu (110100) = 0$ by  (\ref{eqn5:properties}). By induction,
\begin{equation}
\mu (11(01)^k00) = 0
\end{equation}
for all $k \geq 0$.

For $n > 0$, let $A_n = \{\eta: \eta (n) = \eta (n - 1) = 0\}$ and $A = \{A_n~~\text{i.o.}\}$. Then
\begin{equation}
\mu (A) \geq \limsup_{n \rightarrow \infty} \mu (A_n) = \mu (00) > 0.
\end{equation}
Since the event $A$ is an invariant event for the measure $\mu$, by the ergodicity of $\mu$, $\mu (A) = 1$. Define
\begin{equation}
R (\eta) = \inf \{m > 0: \eta (m-1) = \eta (m) = 0 \}.
\end{equation}
Then
\begin{equation}\label{eqn8:properties}
\mu \{ \eta: R (\eta) < \infty \} = 1,
\end{equation}
since $A \subset \{ \eta: R (\eta) < \infty \}$.

Now we are ready to prove (\ref{eqn7:properties}). Note that
\begin{equation}\label{eqn6:properties}
 \mu \{\eta: \eta (0) = \eta (1) = 1\} = \sum_{m = 3}^\infty \mu \{\eta: \eta (0) = \eta (1) = 1, R (\eta) =  m\}.
\end{equation}
Since there must exists a pattern $``11(01)^k00"$ in the event $\{\eta: \eta (0) = \eta (1) = 1, R (\eta) =  m\}$ for some $k$, the right-hand side of (\ref{eqn6:properties}) equals zero and (\ref{eqn7:properties}) holds.
\end{proof}

\section{Proof of Theorem \ref{thm:properties1}}\label{sec3:properties}

The exclusion process can be expressed in terms of the height process (see e.g. \cite{PrahoferSpohn00}). Let $h (t,x)$ be the height process associated to the facilitated exclusion process $\eta_t$:
\begin{equation}
h (t,x) =
\begin{cases}
2 N_t + \sum_{y = 1}^x [1 - 2 \eta_t (y)]  &\text{if $x > 0$},\\
2 N_t   &\text{if $x = 0$},\\
2 N_t - \sum_{y = x+1}^0 [1 - 2 \eta_t (y)]   &\text{if $x < 0$},
\end{cases}
\end{equation}
where $N_t$ is the number of particles across the bond $(0,1)$ during time interval $[0,t)$. At any time $t$, the height process $h (t,x)$ satisfies
\begin{equation}
|h (t,x) - h (t,x-1)| = 1
\end{equation}
for any $x$. The dynamics of the height process is as followings: The height at the site $x$ increases by two at rate one if $h (t,x-2) - h (t,x) = 2$ and $h (t,x+1) - h (t,x) = 1$. The infinitesimal generator of the height process $h(t,x)$ is given by
\begin{equation}
\mathfrak{L} f (h) = \sum_{x \in \mathbb{Z}} 1 \{h (x - 2) = h (x) + 2, h (x + 1) = h (x) + 1\} [f (h^x) - f(h) ],
\end{equation}
where  $h^x (x) = h (x) + 2$ and $h^x (y) = h (y)$ for $y \neq x$.

We say that site $x$ is a record for the height process $h (t,x)$ at time $t$ if
\begin{equation}
h (t,x) \geq h (t,y)~~\text{for all $y < x$}.
\end{equation}
The definition of record for a sequence of random variables appeared early in Example 2.3.2. \cite{Durrett}. We first prove some properties of the records.

\begin{lem}\label{lem:properties1}
(1) If site $x$ is a record at time $t$, then $\eta_t (x) = 0$.
(2) If site $x$ is a record at time $t$, then site $x$ is a record at any time $s \geq t$. In other words, no particles can pass a record.
\end{lem}

\begin{proof}
(1) By the definition of the height process,
\begin{equation}
h (t,x) - h(t,x-1) = 1 - 2 \eta_t (x).
\end{equation}
If $x$ is a record at time $t$, then $h (t,x) \geq h(t,x-1)$. Therefore, $\eta_t (x) = 0$.

(2) Assume that  $x$ is a record at time $t$. We need to show that for all $s \geq t$, for all $y < x$, $h (s,x) \geq h (s,y)$. If this is not the case, let $\tau$ be the first time such that $h (\tau,y) > h (\tau,x)$ for some $y < x$. Then at time $\tau-$, by the dynamics of the height process, $h (\tau-,y-2) = h (\tau-,y) + 2 > h (\tau-,x)$, which is a contradiction.
\end{proof}

\begin{proof}[Proof of Theorem \ref{thm:properties1}]

Suppose that the initial state $\eta_0$ of the exclusion process has distribution $\nu_\rho$. Then $\{ h (0,x), x \in \mathbb{Z} \}$ is a two sided simple random walk with upward probability $1 - \rho$.

(a) If $0 < \rho < 1/2$, then
\begin{equation}
\limsup_{x \rightarrow - \infty} h (0,x) < \infty
\end{equation}
with probability one. Therefore, at time zero, there exists at least one record. Let $x$ be the position of such a record. By Lemma \ref{lem:properties1}, $\eta_t (x) = 0$ for all $t \geq 0$. Since there are finite sites between $x$ and $0$, site $0$ freezes finally.

(b) If $1/2 \leq \rho < 1$, then with probability one,
\begin{equation}\label{eqn9:properties}
\limsup_{x \rightarrow - \infty} h (0,x) = \infty
\end{equation}
and $\eta_0 (x) = 0$ for infinitely many $x > 0$. It suffices to show that
\begin{equation}
\lim_{t \rightarrow \infty} h (t,0) = \infty
\end{equation}
with probability one. Fix integer $M > \max \{h (0,0), h (0,1)\}$. Let $\mathcal{X}_M = \max \{ x < 0: h (0,x) = M \}$. Then $\mathcal{X}_M > - \infty$ by (\ref{eqn9:properties}). Let $\mathcal{T} (x,M) = \inf \{t: h (t,x) \geq M \}$. By the dynamics of the height process, with probability one, $\mathcal{T} (\mathcal{X}_M, M) < \mathcal{T} (\mathcal{X}_M +2, M) < \cdots < \mathcal{T} (1,M) < \infty$ if $\mathcal{X}_M$ is odd, and $\mathcal{T} (\mathcal{X}_M, M) < \mathcal{T} (\mathcal{X}_M +2, M) < \cdots < \mathcal{T} (0,M) < \infty$ if $\mathcal{X}_M$ is even. Therefore,
\begin{equation}
\lim_{t \rightarrow \infty} \max\{h (t,0), h (t,1)\} \geq M.
\end{equation}
Since $|h (t,0) - h (t,1)| = 1$, let $M$ tends to infinity, and the desired result follows.
\end{proof}

\section{Proof of Theorem \ref{thm:properties3}}\label{sec4:properties}

It is well known that the asymmetric simple exclusion process can be mapped into the zero range process (see e.g. \cite{Kipnis86}). Recall that we say there is a hole at site $x$ if $\eta (x) = 0$. We first consider the FTASEP $\eta_t$ with initial distribution  $\bar{\nu}_{1/2} (\cdot) = \nu_{1/2} (\cdot | \eta (0) = 0)$.  At time zero, we label the positions of the holes from the left to the right in an increasing order
$$
\ldots < H_{-1} (0) < H_0 (0) < H_1 (0) < \ldots,
$$
and set $H_0 (0) = 0$. Let $H_i (t)$ is the position of the $i$-th hole at time $t$. Since the ordering of the holes is preserved,
$$
\ldots < H_{-1} (t) < H_0 (t) < H_1 (t) < \ldots,
$$
for all $t$. Let $\xi_t (i) = H_i (t) - H_{i-1} (t) - 1$ be the number of exclusion particles between the $(i-1)$-th hole and the $i$-th hole. Then $\xi_t$ is a zero range process with generator
\begin{equation}
\mathbb{L} f (\xi) = \sum_{x \in \mathbb{Z}} 1 \{\xi (x) > 1 \} [f (\xi^{x,x+1}) - f (\xi)],
\end{equation}
where
\begin{equation}
\xi^{x,x+1} (y) =
\begin{cases}
\xi (x) - 1 &\text{if $y = x$},\\
\xi (x) + 1 &\text{if $y = x+1$},\\
\xi (y)     &\text{otherwise}.\\
\end{cases}
\end{equation}
Note that a $\xi$ particle at site $x$ can jump to its nearest neighbor site if and only if there are at least two particles at site $x$. Without confusion, we will use the letter $\eta_t$ to denote the exclusion process and the letter $\xi_t$ to denote the corresponding zero range process.

Suppose that the FTASEP $\eta_t$ has initial distribution $\bar{\nu}_{1/2}$. Then $\{ \xi_0 (x), x \in \mathbb{Z} \}$ are i.i.d. random variables with Geometric distribution:
\begin{equation}
\mathbf{P}_{\bar{\nu}_{1/2}} (\xi_0 (0) = n) = \frac{1}{2^{n+1}}, n \geq 0.
\end{equation}
Note that under $\bar{\nu}_{1/2}$, the mean of $\xi_0 (0)$ is one. Since $\xi_t (0) \geq 1$ implies $\xi_s (0) \geq 1$ for all $s \geq t$, we expect that the probability converges to one that the origin is exactly occupied  by  one zero range particle.

\begin{lem}\label{lem:properties2}
Suppose that the FTASEP $\eta_t$ has initial distribution $\bar{\nu}_{1/2}$. Then  for every $x \in \mathbb{Z}$,
\begin{equation}\
\lim_{t \rightarrow \infty} \mathbf{P}_{\bar{\nu}_{1/2}} (\xi_t (x) = 1) = 1.
\end{equation}
\end{lem}

\begin{proof}
By the translation invariance of the process and the initial distribution, we just need to consider $x = 0$. We first show that \begin{equation}
\lim_{t \rightarrow \infty} \mathbf{P}_{\bar{\nu}_{1/2}} (\xi_t (0) = 0) = 0.
\end{equation}
Since $\{ \xi_s (0) \geq 1 \} \subset \{ \xi_t (0) \geq 1 \}$ for  $t \geq s$,
\begin{equation}
\lim_{t \rightarrow \infty} \mathbf{P}_{\bar{\nu}_{1/2}} (\xi_t (0) = 0) = \mathbf{P}_{\bar{\nu}_{1/2}} (\text{$\xi_t (0) = 0$ for all $t$}).
\end{equation}
If initially there exists some $N > 0$ such that $\sum_{x = -N}^0 \xi_0 (x) \geq N + 1$, then $\xi_t (0) \geq 1$ for large $t$. Therefore
\begin{equation}
\mathbf{P}_{\bar{\nu}_{1/2}} (\text{$\xi_t (0) = 0$ for all $t$}) \leq \mathbf{P}_{\bar{\nu}_{1/2}} \left( \sum_{x = -N}^{-1} \xi_0 (x) \leq N,~ \forall N \geq 1 \right).
\end{equation}
Since $S_n := \sum_{x = -n }^{-1} (\xi_0 (x) - 1) $ with $S_0 = 0$ is a random walk on $\mathbb{Z}$ with mean zero, by the recurrence of the random walk (see e.g. Theorem 4.2.7 in \cite{Durrett}), the right-hand side of the last inequality equals zero.

We next prove that
\begin{equation}
\lim_{t \rightarrow \infty} \mathbf{P}_{\bar{\nu}_{1/2}} (\xi_t (0) = 1) = 1.
\end{equation}
First note that
\begin{equation}
\frac{d}{d t} \mathbf{E}_{\bar{\nu}_{1/2}} [\xi_t (0)] = \mathbf{P}_{\bar{\nu}_{1/2}} (\xi_t (-1) > 1) - \mathbf{P}_{\bar{\nu}_{1/2}} (\xi_t (0) > 1) = 0,
\end{equation}
then $\mathbf{E}_{\bar{\nu}_{1/2}} [\xi_t (0)] \equiv 1$. Since
\begin{equation}
\begin{split}
\mathbf{E}_{\bar{\nu}_{1/2}} [\xi_t (0)] &\geq \mathbf{P}_{\bar{\nu}_{1/2}} ( \xi_t (0) = 1) + 2 \mathbf{P}_{\bar{\nu}_{1/2}} ( \xi_t (0) \geq 2)\\
&= 2  - \mathbf{P}_{\bar{\nu}_{1/2}} ( \xi_t (0) = 1) - 2 \mathbf{P}_{\bar{\nu}_{1/2}} ( \xi_t (0) = 0),
\end{split}
\end{equation}
then
\begin{equation}
\mathbf{P}_{\bar{\nu}_{1/2}} ( \xi_t (0) = 1) \geq 1 - 2 \mathbf{P}_{\bar{\nu}_{1/2}} ( \xi_t (0) = 0).
\end{equation}
The desired result follows as $t$ tends to infinity.
\end{proof}

We now consider the limiting distribution of the process $\eta_t$ viewed from the tagged hole. Initially put a hole at the origin. Denote by $H_0 (t)$  the position of the hole at time $t$. Let $\bar{\eta}_t := \tau_{H_0 (t)} \eta_t$.

\begin{pro}\label{pro1:properties}
Suppose the initial distribution of $\eta_t$ is $\bar{\nu}_{1/2}$. As $t \rightarrow \infty$, the distribution of the process $\bar{\eta}_t$ at time $t$ converges weakly to the Dirac measure $\delta_{{\bm \eta^0}}$.
\end{pro}

\begin{proof}
Let $A_n (t)$ be the event such that $\bar{\eta}_t (x) = 1$ at odd sites $x$ between the interval $[-2 n, 2 n]$, while $\bar{\eta}_t (x) = 0$ at even sites $x$ between the interval $[-2 n,2 n]$. It suffices to show that
\begin{equation}
\lim_{t \rightarrow \infty} \mathbf{P}_{\bar{\nu}_{1/2}} \left[ A_n (t)  \right] = 1
\end{equation}
for any $n > 0$. The left-hand side of the last formula equals
$$
\lim_{t \rightarrow \infty} \mathbf{P}_{\bar{\nu}_{1/2}} (\text{$\xi_t (x) = 1$ for $-n+1 \leq x \leq n$}),
$$
which is equal to one by Lemma \ref{lem:properties2}.
\end{proof}

\begin{proof}[Proof of Theorem \ref{thm:properties3}]
We first show that there are neither hole pairs nor particle pairs in the system as time $t$ tends to infinity. Note that $\{ (\eta_s (0),\eta_s (1)) \neq (0,0) \}$   implies $\{ (\eta_t (0),\eta_t (1)) \neq (0,0) \}$ for $t \geq s$. Then
\begin{equation}
\mathbf{P}_{\nu_{1/2}} (\eta_t (0) = \eta_t (1) = 0) = \mathbf{P}_{\nu_{1/2}} \left(\eta_t (0) = \eta_t (1) = 0, \eta_0 (0) = 0 \right).
\end{equation}
The right-hand side of  the last formula is equal to
\begin{equation}
\frac{1}{2} \mathbf{P}_{\nu_{1/2}} (\eta_t (0) = \eta_t (1) = 0 | \eta_0 (0) = 0),
\end{equation}
which converges to zero, by Proposition \ref{pro1:properties}, as $t$ tends to infinity.

Let $p_t (11) := \mathbf{P}_{\nu_{1/2}} (\eta_t (x) = \eta_t (x+1) = 1)$, which doesn't depend on $x$ by translation invariance. Similarly define $p_t (10), p_t (01)$ and $p_t (00)$. Since $\mathbf{E}_{\nu_{1/2}} [\eta_t (x)] \equiv 1/2$,
\begin{equation}
   p_t (11) = \frac{1}{2} - p_t (10) = p_t (00) + p_t (10) - p_t (10) = p_t (00).
\end{equation}
Therefore, $p_t (11)$ also converges to zero as $t$ tends to infinity. Then we must have
\begin{equation}
\lim_{t \rightarrow \infty} \nu_{1/2} S(t) = \alpha \delta_{{\bm \eta^0}} + (1 - \alpha) \delta_{{\bm \eta^1}}.
\end{equation}
for some $\alpha \in [0,1]$. Since the density of the particles is $1/2$, $\alpha = 1/2$.
\end{proof}

\section{Discussions}\label{sec5:properties}

In Theorem \ref{thm:properties3}, we considered the limiting distribution of the process in the critical case. In this section, we show how to derive the limiting distribution in the subcritical case, while we are not able to calculate it explicitly.

Suppose the initial distribution of the process $\eta_t$ is $\nu_\rho$ with $0 < \rho < 1/2$. For simplicity, denote by $\nu^\rho_t = \nu_\rho S (t)$ the distribution of the FTASEP at time $t$ when the initial distribution of the process is $\nu_\rho$ and denote by $\nu^\rho_\infty$ the limit of $\nu^\rho_t$ along some subsequence. Since the family of probability measures on the space $\mathbf{X}$ is compact, and the following procedure shows that $\nu^\rho_\infty$ is uniquely determined, it follows that  $\nu^\rho_\infty$ is actually the limit of $\nu^\rho_t$ and is an invariant measure for the process.

Fix a sequence $\lambda_k \in \{0,1\}$ such that
\begin{equation}\label{eqn4:properties}
( \lambda_k, \lambda_{k+1} ) \neq (1,1)
\end{equation}
for all $k$. We now show how to calculate the limiting probability of the event $\{\eta : \eta (x + k) = \lambda_k~\text{for $0 \leq k \leq n - 1$} \}$. By translation invariance, the limiting probability doesn't depend on $x$. We assume (\ref{eqn4:properties}) because, by Theorem \ref{thm:properties1}, $ \nu^\rho_\infty (11) = 0$. We will show how to derive the probability by induction.\\
(1) $n = 1$. By translation invariance, $\nu^\rho_\infty \{\eta : \eta (0) = \lambda_0 \} = \rho^{\lambda_0} (1 - \rho)^{1 - \lambda_0}$.\\
(2) Suppose that we have known the limiting probability when $m < n - 1, n \geq 2$. There are three cases when $m = n - 1$:\\
(a) $\lambda_0 = \lambda_{1} = 0$. Since the event $\{\eta : \eta_t (0) = \eta_t (1) = 0, \eta_t ( k) = \lambda_k~\text{for $2 \leq k \leq n - 1$} \}$ is decreasing in time $t$,
\begin{equation}
\begin{split}
\lim_{t \rightarrow \infty}& \mathbf{P}_{\nu_\rho} ( \eta_t (0) = \eta_t (1) = 0, \eta_t ( k) = \lambda_k~\text{for $2 \leq k \leq n - 1$} ) \\
&= \mathbf{P}_{\nu_\rho} ( \eta_t (0) = \eta_t (1) = 0, \eta_t ( k) = \lambda_k~\text{for $2 \leq k \leq n - 1$ for all $t$} ).
\end{split}
\end{equation}
The right-hand side of the last equation is equal to
\begin{equation}\label{eqn1:properties}
\mathbf{P}_{\nu_\rho} \left( \eta_0 (0) = \eta_0 (1) = 0,  \eta_0 (k) = \lambda_k~\text{ for $2 \leq k \leq n - 1$ and $\sum_{x = -N}^{-1} \eta_0 (x) \leq \frac{N+1}{2}$ for all $N$} \right).
\end{equation}
Let $S_N = \sum_{x = -N}^{-1} (2 \eta_0 (x) - 1)$ with $S_0 = 0$. Under $\mathbf{P}_{\nu_\rho}$, $\{ \eta_0 (x), x \in \mathbb{Z} \}$ are i.i.d.  Bernoulli random variables with parameter $\rho$ and $\mathbf{P}_{\nu_\rho} (S_1 = 1) = \rho = 1 - \mathbf{P}_{\nu_\rho} (S_1 = -1)$. Then (\ref{eqn1:properties}) equals
\begin{equation}
(1 - \rho)^2  \mathbf{P}_{\nu_\rho} (S_N \leq 1~\text{for all $N$}) \prod_{k = 2}^{n - 1}  \rho^{\lambda_k} (1 - \rho)^{1 - \lambda_k} = (1 - 2 \rho) \prod_{k = 2}^{n - 1}  \rho^{\lambda_k} (1 - \rho)^{1 - \lambda_k}.
\end{equation}
(b) $\lambda_0 = 1, \lambda_{1} = 0$. Then
\begin{equation}\label{eqn2:properties}
\begin{split}
\nu^\rho_\infty \{\eta : \eta (k) &= \lambda_{k}~\text{for $0 \leq k \leq n - 1$} \} = \nu^\rho_\infty \{\eta : \eta ( k) = \lambda_{k}~\text{for $1 \leq k \leq n - 1$} \} \\
&- \nu^\rho_\infty \{\eta : \eta (0) = \eta (1) = 0, \eta ( k) = \lambda_{k}~\text{for $2 \leq k \leq n - 1$} \}.
\end{split}
\end{equation}
The first term of the right-hand side of (\ref{eqn2:properties}) involves $n - 1$ sites and, by the assumption, the probability is known. The second term of the right-hand side of (\ref{eqn2:properties}) reduces to case (a). \\
(c) $\lambda_0 = 0, \lambda_{1} = 1$. Since $ \nu^\rho_\infty (11) = 0$,
\begin{equation}
\nu^\rho_\infty \{\eta : \eta ( k) = \lambda_{k}~\text{for $0 \leq k \leq n - 1$} \} = \nu^\rho_\infty \{\eta : \eta ( k) = \lambda_{k}~\text{for $1 \leq k \leq n - 1$} \},
\end{equation}
whose probability is known by the assumption.

\begin{flushleft}
\textbf{Acknowledgements.} \small The authors would like to thank P. A. Ferrari for useful discussions during the Fifth Bath-Beijing-Paris meeting in Beijing during 14-18 May, 2018.
\end{flushleft}

\end{document}